\documentclass[12pt, a4paper]{amsproc}

\usepackage{enumitem}
\usepackage{hyperref}

\newcommand{\V}{{\mathfrak V}}

\newcommand{\Ni}{{\mathfrak N}}

\newcommand{\X}{{\mathfrak X}}
\newcommand{\Y}{{\mathfrak Y}}

\newcommand{\U}{{\mathfrak U}}
\newcommand{\A}{{\mathfrak A}}

\newcommand{\1}{\{1\}}

\newcommand{\var}[1]{\mathrm{var}\left( #1 \right)}
\newcommand{\varr}[1]{\mathrm{var}( #1 )}

\newcommand{\Wr}{\,\mathrm{Wr}\,}
\newcommand{\Wrr}{\,\mathrm{wr}\,}

\theoremstyle{definition}

\theoremstyle{plain}
\newtheorem{Lemma}{\sc Lemma}
\newtheorem{Theorem}{\sc Theorem}

\theoremstyle{remark}
\newtheorem{Remark}{\sc Remark} 
\newtheorem{Example}{\sc Example}

\begin{document}

%%%%%%%%%%%%%%%%%%%%%%%%%%%%%%%%%%%%%%%%%%%%%%
%%%%%%%%%%%%%%%%%%%%%%%%%%%%%%%%%%%%%%%%%%%%%%
$\phantom{Line}$ 
\vskip-4mm

% \address{Informatics and Applied Mathematics Department, Yerevan State University, Yerevan 0025, Armenia.}
% \email{v.mikaelian@gmail.com}


\title[On $K_p$-series and varieties]{\bf On $K_p$-series and varieties generated by wreath products of $p$-groups}

\author[Vahagn H.~Mikaelian]{\bf Vahagn H.~Mikaelian}

\thanks{The author was supported in part by joint grant 15RF-054 of RFBR and SCS MES RA, and by 15T-1A258 grant of SCS MES RA}

\thanks{This article reflects the main plenary  talk at the  \href{http://u.math.biu.ac.il/~vishne/Conferences/Plotkin90/program.html}{\it Groups, Algebras and Identities} conference honoring Professor Boris I.~Plotkin's 90th birthday in Jerusalem, Israel, March 20-24, 2016.}

\date{\today}

\dedicatory{\small To Professor Boris I.~Plotkin on his 90'th anniversary}

\begin{abstract}
Let $A$ be a nilpotent $p$-group of finite exponent, and $B$ be an abelian $p$-groups of finite exponent for a given prime number $p$. Then the wreath product $A \Wr B$ generates the variety $\var{A} \var{B}$ if and only if the group $B$ contains a subgroup isomorphic to the direct product $C_{p^v}^\infty$ of countably many copies of the  cycle $C_{p^v}$ of order $p^v = \exp{(B)}$.
The obtained theorem continues our previous study of cases when $\var{A \Wr B} = \var{A} \var{B}$ holds for some other classes of groups $A$  and $B$ (abelian groups, finite groups, etc.).
\end{abstract}

\maketitle

%%%%%%%%%%%%%%%%%%%%%%%%%%%%%%%%%%%%%%%%%%%%%%
%%%%%%%%%%%%%%%%%%%%%%%%%%%%%%%%%%%%%%%%%%%%%%
%%%%%%%%%%%%%%%%%%%%%%%%%%%%%%%%%%%%%%%%%%%%%%

%%%%%%%%%%%%%%%%%%%%%%%%%%%%%%%%%%%%%%%%%%%%%%
%%%%%%%%%%%%%%%%%%%%%%%%%%%%%%%%%%%%%%%%%%%%%%
%%%%%%%%%%%%%%%%%%%%%%%%%%%%%%%%%%%%%%%%%%%%%%
\section{Introduction and background information}
\label{Introduction}

\noindent
Nilpotent and locally nilpotent subgroups are proven to be very efficient means to study groups.
In particular, the Fitting subgroup and the Frattini subgroup help in study of finite groups, while the Plotkin-Hirsch radical allows to consider infinite groups (see \cite{Fitting Gruppen endlicher Ordnung, Gaschuetz Frattini, Plotking - radical, Plotking - Radical groups, Plotking - Generalized soluble and nilpotent groups, Uber lokal-nilpotente Gruppen} and literature cited therein). 
This trend is naturally inherited by varieties of groups: nilpotent varieties are in some sense simpler varieties (they are irreducible, have finite base rank, their finitely generated groups are factors of finite powers of relatively free groups, etc.), and investigation of nilpotent subvarieties of a variety is an approach to study general varieties.
In the current note we use sequences of nilpotent subvarieties to study varieties generated by non-nilpotent wreath products of groups (here we assume Cartesian wreath products, although the analogs of the statements below also are true for direct wreath products). 

Let us introduce the general context in which we examine varieties generated by wreath products. 
One of the most efficient methods to study product varieties $\U\V$ is finding some groups $A$ and $B$ such that $\U = \var A$, $\V = \var B$, and the wreath product $A \Wr B$ generates $\U\V$, that is, when the equality 
\begin{equation}
\tag{$*$}
\label{EQUATION_main}    
\var{A \Wr B} = \var{A} \var{B}
\end{equation}
holds for the given $A$ and $B$ (as usual, we denote by $\var{X}$ the variety generated by the group $X$). 
Indeed, the product $\U \V$ consists of all possible extensions of all groups $A \in \U$ by all groups $B \in \V$. If \eqref{EQUATION_main} holds for some fixed groups $A$ and $B$, generating the varieties $\U$ and $\V$, then one can restrict to consideration of $\var{A \Wr B}$, which is easier to study rather than to explore all the extensions inside $\U \V$.
Examples of application of this approach are numerous: for earliest results and references see Chapter 2 of Hanna Neumann's monograph~\cite{HannaNeumann} and the work of G.~Baumslag, R.G.~Burns, G.~Higman, C.~Haughton, B.H.~Neumann,  H.~Neumann, P.M.~Neumann~\cite{Burns65, B+3N,  Some_remarks_on_varieties}, etc. 

This motivated our systematic study of equality \eqref{EQUATION_main} for as wide classes of groups as possible. 
In \cite{AwrB_paper}--\cite{Metabelien} we gave a complete classification of all cases, when \eqref{EQUATION_main} holds for {\it abelian} groups $A$ and $B$, and in \cite{wreath products algebra i logika} and \cite{shmel'kin criterion} we fully classified the cases when $A$ and $B$ are any {\it finite} groups.

In the current note we consider the case, when $A$ and $B$ are $p$-groups of finite exponents for a prime $p$, the group  $A$ is nilpotent, and $B$ is abelian. 
Then:

\begin{Theorem}
\label{Theorem wr p-groups}%%%%%%%%%%%%%%%%%%%%%%%%%%%%%%%%%%%%%%%%%%%%%
Let $A$ be a non-trivial nilpotent $p$-group of finite exponent, and $B$ be a non-trivial  abelian group of finite exponent $p^v$ for a prime number $p$. Then the wreath product $A \Wr B$ generates the variety $\var{A} \var{B}$, that is, the variety $\var{A} \A_{p^v}$ if and only if the group $B$ contains a subgroup isomorphic to the direct product $C_{p^v}^\infty$ of countably many copies of the cyclic group $C_{p^v}$ of order $p^v$.
\end{Theorem}

Since $B$ is a non-trivial group of finite exponent, we have $v >0$, and by Pr\"ufer-Kulikov's theorem~\cite{Robinson, Kargapolov Merzlyakov} it is a direct product of copies of some finite cyclic subgroups of prime-power orders. The theorem above states that in this direct product the cycles of order $p^v$ must be present at lest {\it countably many} times, whereas the number of direct summands of orders $p^{v-1}\!, p^{v-2}\!, \ldots, p$ is of no importance.

We below without any definitions  use the basic notions of the theory of varieties of groups such as varieties, relatively free groups, discriminating groups, etc. All the necessary definitions and background information can be found in Hanna Neumann's monograph~\cite{HannaNeumann}. Following the conventional notations, we denote by  ${\sf Q}\X$, ${\sf S}\X$, ${\sf C}\X$ and  ${\sf D}\X$, the classes of all homomorphic images, subgroups, Cartesian products and of direct products of finitely many groups of $\X$ respectively. By Birkhoff's Theorem~\cite{BirkhoffQSC,HannaNeumann} for any class  of groups $\X$ the variety $\var{\X}$ generated by it can be obtained from $\X$ by three operations: $\varr{\X}={\sf QSC}\,\X$.
For information on wreath products we refer to~\cite{HannaNeumann, Kargapolov Merzlyakov}. For the given classes of groups $\X$ and $\Y$ we denote $\X \Wr \Y = \{ X\Wr Y \,|\, X\in \X, Y\in \Y\}$. 
The specific notions, related to $K_p$-series and to nilpotent wreath products can be found in
\cite{Liebeck_Nilpotency_classes,
Meldrum nilpotent, 
Shield nilpotent a,
Shield nilpotent b,
Marconi nilpotent}
and in Chapter 4 of J.D.P.~Meldrum's monograph \cite{Meldrum book}.

When this work was in progress we discussed the topic with Prof. A.Yu.~Ol'shanskii, who suggested ideas for an alternative (very different from our approach) proof for Theorem~\ref{Theorem wr p-groups} using the arguments of his work \cite{Olshanskii Neumanns-Shmel'kin}. In~\cite{classification theorem} we present the outlines of both proofs side by side.

%%%%%%%%%%%%%%%%%%%%%%%%%%%%%%%%%%%%%%%%%%%%%%
%%%%%%%%%%%%%%%%%%%%%%%%%%%%%%%%%%%%%%%%%%%%%%
%%%%%%%%%%%%%%%%%%%%%%%%%%%%%%%%%%%%%%%%%%%%%%
\section{The $K_p$-series and the proof for Theorem~\ref{Theorem wr p-groups}}
\label{k_p series}

In spite of the fact that the soluble wreath products are ``many'' (the wreath product of any soluble group is soluble), the nilpotent wreath products are ``fewer'': as it is proved by G.~Baumslag in 1959, a Cartesian or direct wreath product of non-trivial groups $A$ and $B$ is nilpotent if and only if $A$ is a nilpotent $p$-group of finite exponent, and $B$ is a finite $p$-group~\cite{Baumslag nilp wr}. 

Even after having such an easy-to-use criterion to detect, if the given wreath product is nilpotent, it turned out to be a much harder task and took almost two decades to explicitly compute its nilpotency class in general case. 
H.~Liebeck started by consideration of the cases of wreath products of abelian groups~\cite{Liebeck_Nilpotency_classes}, and the final general formula was found in D.~Shield's work~\cite{Shield nilpotent a, Shield nilpotent b} in 1977. Later the proof was much shortened by R.~Marconi~\cite{Marconi nilpotent}.  

In order to write down the formula we need the notion of $K_p$-series. For the given group $G$ and the prime number $p$ the $K_p$ series $K_{i,p}(G)$ of $G$ is defined for $i=1, 2, \ldots$ by:
\begin{equation}
\label{EQUATION K_p}   
K_{i,p}(G) = 
\!\!\!\!\!\!\!\!\!\!\!
\prod_{\text{$r, j$ with $r p^j \ge i$}} 
\!\!\!\!\!\!\!\!\!\!
\gamma_r(G)^{p^j}\!\!,
\end{equation}
where $\gamma_r(G)$ is the $r$'th term of the lower central series of $G$. 

In particular, $K_{1,p}(G) = G$ holds for any $G$. From definition it is clear that a $K_p$ series is a descending series, although it may not be strictly descending: some of its neighbor terms may coincide. If $G$ is abelian, then in \eqref{EQUATION K_p} the powers $\gamma_1(G)^{p^j}=G^{p^j}$ of the initial term $\gamma_1(G) = G$ need be considered only. 

\begin{Example}
If $G=C_{p^3} \times C_{p} \times C_{p}$ with $p = 5$, then it is easy to calculate that:
$$
K_{1,p}(G) = G; 
\quad
K_{2,p}(G) = \cdots = K_{5,p}(G) \cong C_{p^2}; 
$$
$$
K_{6,p}(G) = \cdots = K_{25,p}(G) \cong C_{p};
\quad
K_{26,p}(G) \cong \{1\}.
$$
\end{Example}

If $G$ is some finite $p$-group, using the $K_p$-series one may introduce the following additional parameters:
let $d$ be the maximal integer such that $K_{d,p}(G) \not= \{1\}$. Then for each $s=1,\ldots, d$ define $e(s)$ by
$$
p^{e(s)} = |K_{s,p} / K_{s+1,p}|,
$$
and set $a$ and $b$ by the rules:
$$
a = 1 + (p-1) \sum_{s=1}^d \left(s \cdot e(s)\vphantom{a^b}\right),
\quad
b = (p-1)d.
$$
The above integer  $d$ does exist, and our notations are correct, for, a finite $p$-group is nilpotent, and its $K_p$-series will eventually reach the trivial subgroup. To keep the notations simpler, the initial group $G$ is not included in the notations of $d$, $e(s)$, $a$ and $b$. But from the context it will always be clear which group $G$ is being considered.

\vskip3mm
Let $A$ be a nilpotent $p$-group of exponent $p^u$, and $B$ be an abelian group of exponent $p^v$, with $u,v >0$. 
Assume all the parameters $d$, $e(s)$, $a$, $b$ are defined specifically for the group $B$.
Then by Shield's formula \cite{Shield nilpotent b} (see also Theorem 2.4 in \cite{Meldrum book}) the nilpotency class of the wreath product $A \Wr B$ is the maximum
\begin{equation}
\label{EQUATION A wr B class}    
\max_{h = 1, \ldots, \, c} \{
a \, h + (s(h)-1)b
\},
\end{equation}
where $s(h)$ is defined as follows: $p^{s(h)}$ is the exponent of the $h$'th term $\gamma_h(A)$ of the lower central series of $A$. 

\begin{Example}
If (again for $p = 5$) $B$ is the group $C_{p^3} \times C_{p} \times C_{p}$ mentioned in previous example, and if $A$ is the group, say, $C_{p^2} \times C_{p}$, then 
$d = 25$; 
$e(1) = 3$, 
$e(2) = e(3) =  e(4) = 0$,  
$e(5) = 1$, 
$e(6) = \cdots =  e(24) = 0$,
$e(25) = 1$, 
$e(26) = 0$;
$a = 133$;
$b = 100$;
$c = 1$ 
and $s(1) = 2$.
Thus, the nilpotency class of the wreath product 
$A \Wr B = (C_{p^2} \times C_{p}) \Wr (C_{p^3} \times C_{p} \times C_{p})$ 
in this case is equal to
$$
a \cdot 1 + (s(1) - 1)b =  133 + (2-1) 100 = 233.
$$
\end{Example}

\vskip3mm
In order to prove Theorem~\ref{Theorem wr p-groups} we will apply Shield's formula to two auxiliary groups.
To construct the first group denote by $\beta$ the cardinality of $B$ and by $A^\beta$ the Cartesian product of $\beta$ copies of $A$. For the given fixed positive integer $l$ and for the integer $t \ge l$ introduce the group $Z(l,t)$ as the direct product:
\begin{equation}
\label{EQUATION A Wr Z definition}    
Z(l,t) =\underbrace{C_{p^v} \times \cdots \times C_{p^v}}_l  
\,\, \times \,\,
\underbrace{C_{p^{v-1}} \times \cdots \times C_{p^{v-1}}}_{t-l}.
\end{equation}

\begin{Lemma}
\label{LEMMA bound for A beta Wr Z}%%%%%%%%%%%%%%%%%%%%%%%%%%%%%%%%%%%%%%%%%%%%%
Assume $A$, $B$ and $\beta$ are defined as above and $l$ is any positive integer. Also, assume the exponent of  $\gamma_c(A)$ is $p^\alpha$ ($\alpha \not= 0$, since the class of $A$ is c). Then there is a positive integer $t_0$ such that for all $t > t_0$ the nilpotency class of the wreath product
$A^\beta \Wr Z(l,t)$ is equal to
\begin{equation}
\label{EQUATION bound 1}    
c + c\,t(p-1)\left( 
\frac{1 - p^{v-1}}{\hskip-4mm 1-p} +l/t \cdot p^{v-1}
\right)
+ (\alpha-1)(p-1)p^{v-1}.
\end{equation}
\end{Lemma}

\begin{proof}
Denote $Z = Z(l,t)$ and notice that $A^\beta \Wr Z$ is nilpotent by Baumslag's theorem~\cite{Baumslag nilp wr}. Let us compute the $K_{p}$-series for $Z$ and, to keep the notations simpler, not include in the formulas the underbraces of \eqref{EQUATION A Wr Z definition} with $l$ and $t-l$.
For $i=1$ we have $K_{1,p}(Z) = Z$. 
For $i=2, \ldots, p$ we get:
$$
K_{i,p}(Z) = Z^p =C_{p^{v-1}} \times \cdots \times C_{p^{v-1}}  
\,\, \times \,\,\, 
C_{p^{v-2}} \times \cdots \times C_{p^{v-2}}.
$$
For $i=p^k+1, \ldots, p^{k+1}$ we have:
$$
K_{i,p}(Z) = Z^{p^{k+1}} =C_{p^{v-(k+1)}} \times \cdots \times C_{p^{v-(k+1)}}  
\,\, \times \,\,\, 
C_{p^{v-(k+2)}} \times \cdots \times C_{p^{v-(k+2)}}.
$$
In particular, for $i=p^{v-3}+1, \ldots, p^{v-2}$ we get:
$$
K_{i,p}(Z) = Z^{p^{v-2}} =C_{p^2} \times \cdots \times C_{p^2}  
\,\, \times \,\,\, 
C_{p} \times \cdots \times C_{p},
$$
for $i=p^{v-2}+1, \ldots, p^{v-1}$ we get:
$$
K_{i,p}(Z) = Z^{p^{v-1}} =C_{p} \times \cdots \times C_{p}  
\quad \text{(just $l$ factors)},
$$
and, finally, for $i=p^{v-1}+1$ the series terminates on $K_{i,p}(Z) = \1$.

Therefore, $d = p^{v-1}$ and all the parameters $e(i)$ are zero except the following ones:
$$
e(1) = e(p) = \cdots = e(p^{u-2}) = t
\quad
\text{and}
\quad
e(p^{u-1}) = l.
$$
Thus:
$$
a = 1 + (p-1)(t + p t + \cdots + p^{v-2} t + p^{v-1} l) 
$$
$$
= 1 + t(p-1)\left( 
\frac{1 - p^{v-1}}{\hskip-4mm 1-p} +l/t \cdot p^{v-1}
\right)
$$
and
$$
b = (p-1)d=(p-1)p^{v-1}.
$$

To deal with the parameters $s(h)$, $h = 1, \ldots, c$, for $A^\beta$ notice that $\gamma_h(A^\beta)$ is a subgroup of the Cartesian power $\left(\gamma_h(A)\right)^\beta$ (they may not be equal if  $\beta$ is infinite) and, on the other hand, $\gamma_h(A^\beta)$ contains elements with exponent equal to the exponent of $\left(\gamma_h(A)\right)^\beta$. Therefore, the exponents of $\gamma_h(A^\beta)$ and of $\gamma_h(A)$ are equal for all $h=1,\ldots,c$, and the parameters $s(h)$ are the same for both $\gamma_h(A^\beta)$ and of $\gamma_h(A)$. 

By Shield's formula the nilpotency class of $A^\beta \Wr Z$ is the maximum of values 
\begin{equation}
\label{EQUATION two summands}    
a \, h + (s(h)-1)b, \quad h = 1, \ldots, c.
\end{equation}

In spite of the fact that having a larger $h$ we get a larger summand $a h$, it may turn out that for some $h < c$ the exponent of $\gamma_h(A)$ is so much larger than the exponent of $\gamma_c(A)$ that for the given $t$  the highest value of \eqref{EQUATION two summands} is achieved not for $h = c$ (examples are easy to build). However, the  second summand in \eqref{EQUATION two summands} may get just $c$ distinct values not dependent on $t$, whereas the first  summand includes $a$, which grows infinitely and monotonically with $t$. Thus, even with a fixed $l$ there is a large enough $t_0$ such that the maximal value of \eqref{EQUATION two summands} is achieved with $h = c$ for all $t \ge t_0$. 
To finish this proof just recall that we denoted the exponent of $\gamma_c(A)$ by $p^\alpha$.
\end{proof}

\vskip2mm
To introduce our second group we need a finitely generated (and, in fact, also finite, since it is in a locally finite variety) subgroup $\tilde A$ of $A$ such that the exponents of terms $\gamma_h(\tilde A)$ and $\gamma_h(A)$ are equal for each $h = 1, \ldots, c$. Clearly, the nilpotency classes of $\tilde A$ and of $A$ will then be equal. 

Notice that each term $\gamma_h(A)$ contains such an element $a_h$, the exponent of which is equal to the exponent $p^{s(h)}$ of $\gamma_h(A)$. This is possible, since $A$ is a $p$-group of finite exponent. Since $a_h$ is an element of the verbal subgroup $\gamma_h(A)$ for the word $\gamma_h(x_1,\ldots, x_h)$, there are some finitely many elements $a_{h,1}, \ldots, a_{h,r_h}\in A$ such that 
$a_h \in \gamma_h \left(\langle a_{h,1}, \ldots, a_{h,r_h}\rangle \right)$. Collecting these finitely many generators for all $h$, we get the group 
$$
\tilde A = \langle a_{h,i}, \ldots, a_{h,r_h} \,|\, h = 1, \ldots, c \, \rangle,
$$
which does have the property we needed. Assume $\tilde A$ is a $z$-generator group
%%  and recall, that we denoted the exponent of $\gamma_c(A)=\gamma_c(\tilde A)$ by $p^\alpha$.
and denote by $Y(z,t)$ the product:
\begin{equation}
\label{EQUATION A Wr Y definition}    
Y(z,t) =\underbrace{C_{p^v} \times \cdots \times C_{p^v}}_{t-z}.
\end{equation}

Then we have the following value for the nilpotency class of the wreath product $\tilde A \Wr Y(z,t)$:

\begin{Lemma}
\label{LEMMA bound for tilde A Wr Z}%%%%%%%%%%%%%%%%%%%%%%%%%%%%%%%%%%%%%%%%%%%%%
Assume $A$, $\tilde A$, $z$  and $\alpha$ are defined as above. Then there is a positive integer $t_1$ such that for all $t > t_1$ the nilpotency class of the wreath product
$\tilde A \Wr Y(z,t)$  is equal to
\begin{equation}
\label{EQUATION bound 2}    
c + c(t-z)(p-1) \frac{1 - p^{v}}{\hskip-1mm 1-p}
+ (\alpha-1)(p-1)p^{v-1}.
\end{equation}
\end{Lemma}

\begin{proof}
Denote $Y = Y(z,t)$ and notice that $\tilde A \Wr Y$ is nilpotent. 
Let us compute the $K_{p}$-series for $Y$ by the same routine procedure as in previous proof.
For $i=1$ we have $K_{1,p}(Y) = Y$. 
For $i=2, \ldots, p$ we have:
$$
K_{i,p}(Y) = Y^p =C_{p^{v-1}} \times \cdots \times C_{p^{v-1}}.
$$
For $i=p^k+1, \ldots, p^{k+1}$ we have:
$$
K_{i,p}(Y) = Y^{p^{k+1}} =C_{p^{v-(k+1)}} \times \cdots \times C_{p^{v-(k+1)}}.
$$
In particular, for $i=p^{v-3}+1, \ldots, p^{v-2}$ we get:
$$
K_{i,p}(Y) = Y^{p^{v-2}} =C_{p^2} \times \cdots \times C_{p^2},
$$
for $i=p^{v-2}+1, \ldots, p^{v-1}$ we get:
$$
K_{i,p}(Y) = Y^{p^{v-1}} =C_{p} \times \cdots \times C_{p},
$$
and, finally, for $i=p^{v-1}+1$ we get $K_{i,p}(Y) = \1$.

Again, $d = p^{v-1}$ and the only non-zero parameters $e(i)$ are:
$$
e(1) = e(p) = \ldots = e(p^{u-1}) = t-z.
$$
Thus:
$$
a = 1 + (p-1)\Large\left( (t-z) + p (t-z) + \cdots + p^{v-1} (t-z) \Large\right) 
$$
$$
= 1 + (t-z)(p-1) 
\frac{1 - p^{v}}{\hskip-1mm 1-p}
$$
and
$$
b = (p-1)d=(p-1)p^{v-1}.
$$

Here the parameters $s(h)$, $h = 1, \ldots, c$,  are the same for $\tilde A$ and $A$, so the nilpotency class of $\tilde A \Wr Y$ is, like in previous proof, the maximum of values 
\begin{equation}
\label{EQUATION two summands second}    
a \, h + (s(h)-1)b, \quad h = 1, \ldots, c,
\end{equation}
where $p^{s(h)}$ is the exponent of $\gamma_h(A)$. 
This exponent for some $h < c$ may be so much larger than the exponent of $\gamma_c(A)$ that for the given $t$  the highest value of \eqref{EQUATION two summands second} is achieved not for $h = c$. However, there is a large enough $t_1$ such that the maximal value of \eqref{EQUATION two summands second} is achieved with $h = c$ for all $t \ge t_1$. Thus, we can assume $h=c$ in formula \eqref{EQUATION two summands second} with $s(c)  = \alpha$.
\end{proof}

\begin{Remark}
Clearly, it would be possible to compute the exact values for the limits $t_0$ and $t_1$ in Lemma~\ref{LEMMA bound for A beta Wr Z} and Lemma~\ref{LEMMA bound for tilde A Wr Z}. However, that would bring nothing but lengthy calculations, since the exact values of $t_0$ and $t_1$ are not relevant for the rest.
\end{Remark}

%%%%%%%%%%%%%%%%%%%%%%%%%%%%%%%%%%%%%%%%%
\vskip2mm

Before we proceed to the proof of Theorem~\ref{Theorem wr p-groups} let us bring two technical lemmas, where we group a few facts, which either are known in the literature, or are proved by us earlier (see Proposition 22.11 and Proposition 22.13 in \cite{HannaNeumann}, Lemma 1.1 and Lemma 1.2 in \cite{AwrB_paper} and also \cite{ShmelkinOnCrossVarieties}). The proofs can be found in  \cite{AwrB_paper}, and we bring these lemmas here without arguments:

\begin{Lemma}
\label{X*WrY_belongs_var}%%%%%%%%%%%%%%%%%%%%%%%%%%%%%%%%%%%%%%%%%%%%%
For arbitrary classes $\X$ and $\Y$ of groups and for arbitrary groups $X^*$ and $Y$, where either $X^*\in 
{\sf Q}\X$, or $X^*\in {\sf S}\X$, or $X^*\in {\sf C}\X$, and where $Y\in \Y$, the group $X^* \Wr
Y$  belongs to the variety $\var{\X \Wr \Y}$.
\end{Lemma}

\begin{Lemma}
\label{XWrY*_belongs_var}%%%%%%%%%%%%%%%%%%%%%%%%%%%%%%%%%%%%%%%%%%%%%
For arbitrary classes $\X$ and $\Y$ of groups and for arbitrary groups $X$ and $Y^*$, where $X\in\X$ and 
where $Y^*\in {\sf S}\Y$, the group $X \Wr Y^*$  belongs to the variety $\var{\X \Wr \Y}$.
Moreover, if  $\X$ is a class of abelian groups, then for each  $Y^*\in {\sf Q}\Y$ the group $X
\Wrr Y^*$  also belongs to $\var{\X \Wr \Y}$.
\end{Lemma}

Now we can prove the main statement:

\begin{proof}[Proof of Theorem~\ref{Theorem wr p-groups}]
That the condition of the theorem is sufficient is easy to deduce from the discriminating properties of $C_{p^v}^\infty$ (see \cite{B+3N} or Corollary 17.44 in \cite{HannaNeumann}). Since $B$ and $C_{p^v}^\infty$ generate the same variety $\A_{p^v}$ and, since $C_{p^v}^\infty$ is isomorphic to a subgroup of $B$, then by \cite[17.44]{HannaNeumann} $B$ also discriminates $\A_{p^v}$. It remains to apply Baumslag's theorem: since $B$ discriminates $\var{B} =  \A_{p^v}$, the wreath product $A \Wr B$ discriminates and, thus, generates the product $\var{A} \A_{p^v}$ (see  \cite{B+3N} or the statements 22.42, 22.43, 22.44 in \cite{HannaNeumann}).

Turning to the proof of necessity of the condition suppose the group $B$ contains no subgroup isomorphic to $C_{p^v}^\infty$. By Pr\"ufer-Kulikov's theorem~\cite{Robinson, Kargapolov Merzlyakov} $B$ is a direct product of some (probably infinitely many) finite cyclic subgroups, the orders of which all are some powers of $p$. Since $B$ is of exponent $p^v$, all these orders are bounded by $p^v$, and there is at least one factor isomorphic to $C_{p^v}$. By assumption, there are only finitely many, say $l$, such factors, and collecting them together, we get $B = B_1 \times B_2$, where  
$$
\text{$B_1 = C_{p^v} \times \cdots \times C_{p^v}$    \quad     ($l$ factors),}
$$ 
and where $B_2$ is a direct product of some cycles of orders not higher than $p^{v-1}$.

Take an arbitrary $t$-generator group $G$ in variety $\varr{A \Wr B}$. By \cite[16.31]{HannaNeumann} $G$ is in variety generated by all the $t$-generator subgroups of $A \Wr B$.
Assume $T$ is one of such $t$-generator subgroups and denote by $H$ its intersection with the base subgroup $A^B$ of $A \Wr B$. Then 
$$
T / H \cong (T A^B) / A^B \le (A \Wr B)/ A^B \cong B
$$ and, thus, $T$ is an extension of $H$ by means of an at most $t$-generator subgroup $B'$ of $B= B_1 \times B_2$. By Kaloujnine-Krasner's theorem \cite{KaloujnineKrasner} the group $G$ is embeddable into $H \Wr B'$ (see also~\cite{Ol'shanskii Kaluzhnin - Krasner}). 

The group $B'$ is a direct product of at most $t$ cycles, of which at most $l$ cycles are of order $p^v$, and the rest are of strictly lower orders. So $B'$ is isomorphic to a subgroup of $Z(l,t)$ for a suitable $t$. Since $H$ is a subgroup in $A^\beta$, we can apply Lemma~\ref{X*WrY_belongs_var} and Lemma~\ref{XWrY*_belongs_var} to get that 
$$
H \Wr B \in \var{A^\beta \Wr Z(l,t)}.
$$
According to Lemma~\ref{LEMMA bound for A beta Wr Z} we get that the nilpotency class of $H \Wr B$ and of $T$ are bounded by formula \eqref{EQUATION bound 1} for all $t > t_0$. 

\vskip3mm
Our proof will be completed if we discover a $t$-generator  group in $\var{A} \var{B}$, with nilpotency class higher than  \eqref{EQUATION bound 1}, at lest for some  $t$. 

The group  $\tilde A \Wr Y(z,t)$ of Lemma~\ref{LEMMA bound for tilde A Wr Z} is $t$-generator, because $\tilde A$ is a $z$-generator group. For sufficiently large $t > t_1$ the nilpotency class of this group is given by formula \eqref{EQUATION bound 2}. To compare the values of \eqref{EQUATION bound 1}  and \eqref{EQUATION bound 2} notice that they both consist of three summands, of which the first and the third are the same in both formulas. Let us compare the second summands in \eqref{EQUATION bound 1} and in \eqref{EQUATION bound 2}. After we eliminate the common multiplier $p-1$ in both of them, we have:
\begin{equation}
\label{EQUATION compare 1}    
c\,t\left( 
\frac{1 - p^{v-1}}{\hskip-4mm 1-p} +l/t \cdot p^{v-1}
\right) 
= c\,t \, \frac{1 - p^{v-1}}{\hskip-4mm 1-p} + c \, l \, p^{v-1}
\end{equation}
and 
\begin{equation*}
\label{EQUATION compare 2} 
c(t-z) \frac{1 - p^{v}}{\hskip-1mm 1-p}  = c(t-z) \left(  \frac{1 - p^{v-1}}{\hskip-4mm 1-p} + p^{v-1} \right) 
\end{equation*}
\begin{equation}
\label{EQUATION compare 3} 
=  c\,t \frac{1 - p^{v-1}}{\hskip-4mm 1-p} 
+ c\,t \, p^{v-1} 
- cz \frac{1 - p^{v-1}}{\hskip-4mm 1-p}
- c z \, p^{v-1}.
\end{equation}
The summand $c\,t \, \frac{1 - p^{v-1}}{\hskip-4mm 1-p}$ is the same in \eqref{EQUATION compare 1}  and  \eqref{EQUATION compare 3}, so we can eliminate it also, and just compare the remaining expressions: 
\begin{equation}
\label{EQUATION compare 4} 
c \, l \, p^{v-1}
\quad \text{and} \quad
c\,t \, p^{v-1} 
- cz \frac{1 - p^{v-1}}{\hskip-4mm 1-p}
- c z \, p^{v-1}.
\end{equation}
Since $c, l$ and $v$ are fixed, the left-hand side of \eqref{EQUATION compare 4} is a positive constant. Since $z$ also is fixed, the second and third summands on the right-hand side of \eqref{EQUATION compare 4} are some negative constants, which make the sum smaller. But, whatever these negative constants be, the other summand $c\,t \, p^{v-1} $ in \eqref{EQUATION compare 4} grows infinitely as $t$ grows. 

So for sufficiently large $t^*$ the value of \eqref{EQUATION compare 3}  is larger than that of  \eqref{EQUATION compare 1}  (if necessary, we may also take $t^* > t_0, \, t_1$).
Thus, the nilpotency class of the $t$-generator group  $\tilde A \Wr Y(z,t)$ from the variety $\var{A} \var{B}$ is higher than the maximum of the nilpotency classes of the $t$-generator groups in $\var{A \Wr B}$ for all  $t > t^*$. So $\tilde A \Wr Y(z,t)$ does not belong to $\var{A \Wr B}$, and the proof of the theorem is completed.
\end{proof}

It would not be hard to compute the exact value for $t^*$ in the proof above. We omit it to avoid routine calculations. 

\begin{Remark}
The reader may compare the proofs in this note with proofs in Section 4 in~\cite{AwrB_paper} or Section 6  in~\cite{Metabelien}, where we considered similar problem for wreath products of {\it abelian} $p$-groups. That time we used the specially defined functions $\lambda (A, B, t)$, and for bounds on nilpotency classes of wreath products of abelian groups we applied Liebeck's formula~\cite{Liebeck_Nilpotency_classes}. As one may notice, in \cite{AwrB_paper, Metabelien} we had a by far simpler situation than what we discussed in Lemma~\ref{LEMMA bound for A beta Wr Z} and  Lemma~\ref{LEMMA bound for tilde A Wr Z}.
\end{Remark}

Turning to the examples of usage of Theorem~\ref{Theorem wr p-groups} notice that Example 4.6 in~\cite{AwrB_paper}  and Example 6.9 in~\cite{Metabelien} already are illustrations of Theorem~\ref{Theorem wr p-groups}, since they consider wreath products of abelian $p$-groups of finite exponents. As an example with a nilpotent, non-abelaian passive group we may consider:

\begin{Example}The dihedral group $D_4$ is of nilpotency class $2$. Its order is $8$ and the exponent is $4=2^2$. L.G.~Kov{\'a}cs in \cite{Kovacs dihedral} has computed the variety it generates: $\var {D_4}= \A_2^2 \cap \Ni_2$. That for any finite $2$-group $B$ the wreath product $D_4 \Wr B$ does not generate the product $\var{D_4} \var {B} = (\A_2^2 \cap \Ni_2 ) \var {B}$ is clear from the fact that $D_4 \Wrr B$ is a nilpotent group, whereas no product variety may be nilpotent (if both factors are non-trivial). Now take $B$ to be an infinite abelian group of exponent, say, $2^v$. By Theorem~\ref{Theorem wr p-groups}  
$$
D_4 \Wr B = (\A_2^2 \cap \Ni_2 ) \var {B} = (\A_2^2 \cap \Ni_2 ) \A_{p^v}
$$ 
holds if and only if $B$ contains a subgroup isomorphic to $C_{p^v}^\infty$. In particular, if 
$$
B = \underbrace{C_{p^v} \times \cdots \times C_{p^v}}_{n}   
\,\, \times \,\, 
\underbrace{C_{p^{v-1}} \times \cdots \times C_{p^{v-1}} \times \cdots}_{\infty},
$$
then $D_4 \Wr B$ does not generate $(\A_2^2 \cap \Ni_2 ) \A_{p^v}$. Moreover, it will not generate it even if we add to $B$ the ``large'' direct factor 
$$
\underbrace{C_{p^{v-2}} \times \cdots \times C_{p^{v-2}} \times \cdots}_{\infty}
\,\, \times \cdots \times \,\, 
\underbrace{C_{p} \times \cdots \times C_{p} \times \cdots}_{\infty}
.
$$
\end{Example}

The quaternion group $Q_8$ of order eight generates the same variety as $D_4$ (see~\cite{HannaNeumann}), and it also is nilpotent of class $2$. So a similar example can be constructed for this group also.

\vskip3mm

{\small
\noindent
% Informatics and Appl.~Mathematics Dept., 
Yerevan State University, Alex Manoogian 1, Yerevan 0025, Armenia. Email: vmikaelian@ysu.am.
\vskip2mm 

\noindent
% CSE, 
American University of Armenia, 40 Marshal Baghramyan Ave., Yerevan 0019, Armenia. Email: vmikaelian@aua.am.
}


\vskip10mm

\begin{thebibliography}{88}
\parskip1mm

\bibitem{Baumslag nilp wr}
G.~Baumslag,
\textit{Wreath products and $p$-groups},
Proc.~Camb.~Philos.~Soc.~55 (1959), 224--231.


\bibitem{B+3N}
G.~Baumslag, B.H.~Neumann,  Hanna Neumann, P.M.~Neumann
\textit{On varieties generated by finitely generated group},
Math.~Z., 86 (1964), 93--122.

\bibitem{BirkhoffQSC}
G.~Birkhoff,
\textit{On the structure of abstract algebras},
Proc. Cambridge Phil. Soc., 31 (1935), 433--454.

\bibitem{Burns65}
 R.G.~Burns,
{\it Verbal wreath products and certain product varieties of groups} 
J. Austral. Math. Soc.  7 (1967), 356--374.

\bibitem{Fitting Gruppen endlicher Ordnung}
H.~Fitting,
{\it Beitraege zur Theorie der Gruppen endlicher Ordnung} (German),
Jahresber. Dtsch. Math.-Ver. 48 (1938), 77--141.

 
\bibitem{Gaschuetz Frattini}
W.~Gasch\"utz,
{\it \"Uber die $\Phi$-Untergruppe endlicher Gruppen} (German),
Math. Z. 58 (1953), 160--170.


\bibitem{Some_remarks_on_varieties}
G.~Higman,
\textit{Some remarks on varieties of groups},
Quart.~J.~Math.~Oxford, (2) 10 (1959), 165--178.


\bibitem{Uber lokal-nilpotente Gruppen}
K.A.~Hirsch,
\textit{\"Uber lokal-nilpotente Gruppen} (German),
Math. Z.  63  (1955), 290--294.

\bibitem{KaloujnineKrasner}
L.~Kaloujnine, M.~Krasner,
\textit{Produit complete des groupes de permutations et le probl\`eme d'extension des groupes, III},
Acta Sci. Math. Szeged, 14 (1951), 69--82.

\bibitem{Kargapolov Merzlyakov}
M.I. Kargapolov, Yu.I. Merzlyakov,
\textit{Fundamentals of group theory} 4th ed. (Russian), Moscow, Nauka, Fizmatlit 1996.
\hskip3mm
English translation from the 2nd ed.~by Robert G.~Burns, Graduate Texts in Mathematics 62, New York-Heidelberg-Berlin: Springer-Verlag, XVII 1979.


\bibitem{Kovacs dihedral}
L.G.~Kov{\'a}cs,
\textit{Free groups in a dihedral variety},
Proc. Roy. Irish Acad. Sect. A  89  (1989), 1, 115--117.


\bibitem{Liebeck_Nilpotency_classes}
H.~Liebeck,
\textit{Concerning nilpotent wreath products},
Proc.~Cambridge Phil.~Soc., 58 (1962), 443--451.


\bibitem{Marconi nilpotent}
R.~Marconi,
\textit{On the nilpotency class of wreath products} (Italian),
Boll.~Unione Mat.~Ital., VI.~Ser., D, Algebra Geom. 2 (1983), 1, 9--20.

\bibitem{Meldrum nilpotent}
J.D.P.~Meldrum,
\textit{On nilpotent wreath products},
Proc.~Cambridge Philos.~Soc., 68 (1970), 1--15.

\bibitem{Meldrum book}
J.D.P.~Meldrum,
\textit{Wreath products of groups and semigroups},
Pitman Monographs and Surveys in Pure and Applied Mathematics, 74, Harlow, Essex: Longman Group Ltd.~xii 1995.





%%%%%%%%%%%%%%%%%%%%%%%%%%%%%%%%%%%%%%%%%%%%%%

% \bibitem{SubnormalEmbeddingTheorems}
% V. H. Mikaelian, {\it Subnormal embedding theorems for groups,}
% J. London Math. Soc., 62 (2000), 398--406.

\bibitem{AwrB_paper}
V.H.~Mikaelian,
\textit{On varieties of groups generated by wreath products of abelian groups}, 
Abelian groups, rings and modules (Perth, Australia, 2000), Contemp. Math., 273, Amer. Math. Soc., Providence, RI (2001), 223--238.

\bibitem{finitely generated abelian}
V.H.~Mikaelian,
\textit{On wreath products of finitely generated abelian groups}, 
Advances in Group Theory, Proc. Internat. Research Bimester dedicated to the memory of Reinhold Baer, (Napoli, Italy, May-June, 2002), Aracne, Roma, 2003, 13--24.

% \bibitem{Two problems}
% V.H.~Mikaelian,
% \textit{Two problems on varieties of groups generated by wreath products of groups},  
% Int. J. Math. Sci.  31  (2002), no. 2, 65--75. 

\bibitem{Metabelien}
V.H.~Mikaelian,
\textit{Metabelian varieties of groups and wreath products of abelian groups}, 
Journal of Algebra, 2007 (313), 2, 455--485. 

% \bibitem{wreath products Prilozh}
% V.H.~Mikaelian,
% \textit{Varieties generated by wreath products of abelian groups}, 
% translated from Sovrem. Mat. Prilozh., Vol. 83, 2012. J. Math. Sci. (N. Y.)  195  (2013), no. 4, 523--528.

\bibitem{wreath products algebra i logika}
V.H.~Mikaelian,
\textit{Varieties generated by wreath products of abelian and nilpotent groups} (Russian and English), 
Algebra i Logika,  54 (2015), 1, 103--108. Translated in Algebra and Logic, 54 (2015), 1, 70-73.

\bibitem{shmel'kin criterion}
V.H.~Mikaelian,
\textit{The Criterion of Shmel'kin and Varieties Generated by Wreath Products of Finite Groups}, 
accepted in Algebra i Logika.

\bibitem{classification theorem}
V.H.~Mikaelian,
\textit{On classification of varieties generated by wreath products}, submitted.





\bibitem{abelian subgr in metab gr}
V.H. Mikaelian, A.Yu. Olshanskii
{\it  On abelian subgroups of finitely generated metabelian groups}, 
Journal of Group Theory, {\bf 16} (2013), 695--705.

%%%%%%%%%%%%%%%%%%%%%%%%%%%%%%%%%%%%%%%%%%%%%%


\bibitem{HannaNeumann}
Hanna Neumann, 
{\it  Varieties of Groups}, 
Ergebn.~Math.~Grenzg., 37, Berlin-Heidelberg-New York, Springer-Verlag 1967.



\bibitem{Olshanskii Neumanns-Shmel'kin}
A.Yu.~Olshanskii
\textit{The Neumanns-Shmel'kin's theorem (Russian)}, 
Vestnik Mosk. Univ., Ser. Matem.(1986), 6, 61--64.

\bibitem{Ol'shanskii Kaluzhnin - Krasner}
A.Yu. Olshanskii, 
\textit{On Kaluzhnin-Krasner's embedding of groups,} 
Algebra Discrete Math. 19 (2015), 1, 77--86.

\bibitem{Plotking - radical}
B.I.~Plotkin,
\textit{On some criteria of locally nilpotent groups} (Russian), 
Uspehi Mat. Nauk 9 (1954), 181--186. Translated in Amer. Math. Soc. Translations (2) 17 (1961), 1--7.


\bibitem{Plotking - Radical groups}
B.I.~Plotkin,
\textit{Radical groups} (Russian), 
 Mat. Sb. 37 (1955), 507--526. Translated in Amer. Math. Soc. Translations (2) 17 (1961), 9--28.

\bibitem{Plotking - Generalized soluble and nilpotent groups}
B.I.~Plotkin,
\textit{Generalized soluble and nilpotent groups} (Russian), 
Uspehi Mat. Nauk 13 (1958), 89-172. Translated in Amer. Math. Soc. Translations (2) 17 (1961), 29--115.


\bibitem{Robinson}
D.J.S.~Robinson, 
\textit{A Course in the Theory of Groups}, second edition, 
Springer-Verlag, New York, Berlin, Heidelberg 1996.

\bibitem{Roman’kov nilpotent}
V.A.~Roman'kov, 
{\it Embedding theorems for nilpotent groups} (Russian), 
Sibirsk.~Mat.~Zh.~13 (1972), 859–-867.

\bibitem{Shield nilpotent a}
D.~Shield,
\textit{Power and commutator structure of groups}, 
Bull.~Austral.~Math.~Soc., 17 (1977), 1--52.

\bibitem{Shield nilpotent b}
D.~Shield,
\textit{The class of a nilpotent wreath product}, 
Bull.~Austral.~Math.~Soc., 17 (1977), 53--89.



\bibitem{ShmelkinOnCrossVarieties}
A.L.~Shmel'kin,
\textit{Wreath products and varieties of groups} (Russian),
Izv.~AN SSSR, ser.~matem., 29 (1965), 149--170.
Summary in English: Soviet Mathematics, 5 (1964), 4. 
Translation of Dokl.~Akad.~Nauk S.S.S.R.~for Am.~Math.~Soc.

\end{thebibliography}
\end{document}